\documentclass[12pt]{article}
\usepackage{amsmath}
\usepackage{amssymb,amscd}
\usepackage{amscd}
\usepackage{amsmath,latexsym,amssymb,amsthm}
\usepackage{graphicx}
\usepackage{hyperref}

\theoremstyle{plain}
\newtheorem{thm}[subsection]{Theorem}
\newtheorem{lem}[subsection]{Lemma}

\newtheorem{conj}{Conjecture}

\theoremstyle{definition}

\newtheorem{defn}[subsection]{Definition}
\newcommand{\card}[1]{\ensuremath{\left |#1\right |}}
\newcommand{\set}[1]{\ensuremath{\left \{ #1 \right \} } }

\def\v{\vee}
\def\w{\wedge}

\def\L{\mathcal{L}}

\def\a{\alpha}

\def\b{\beta}

\def\d{\delta}
\def\g{\gamma}

\def\v{\vee}
\def\w{\wedge}

\def\L{\mathcal{L}}

\def\a{\alpha}

\def\b{\beta}

\def\d{\delta}
\def\g{\gamma}

\title{On the number of non-comparable pairs of elements in a distributive lattice }
\author{
        Himadri Mukherjee\\
                Department of Mathematics and Statistics\\
        Indian Institute of Science Education and Research, Kolkata\\
        himadri@iiserkol.ac.in
       }
\date{\today}

\begin{document}

\maketitle

\begin{abstract}
In this article we introduce the study of the number of pairs of non-comparable elements in a distributive lattice $\L$ as an invariant of the lattice. We give several tight lower and upper bounds for the number. Using the bounds we describe the lattices precisely for which the algebraic varieties associated are complete intersections.
\end{abstract}

\section{Introduction}
\noindent
Motivated by the work of Lakshmibai and Gonciulea in \cite{glnext,g-l} we look at the number of non-comparable elements in a distributive lattice. In the paper \cite{glnext} the authors used the number of a special types of non-comparable elements in the Young lattice $I_{d,n}$ to find the singular locus of the Schubert variety $X_{d,n}$. In the work of Lakshmibai, Brown in \cite{b-l, GH} and in the work of the present author in \cite{HL, thesis} similar techniques of counting certain types of non-comparable pairs of elements is used to write the singular locus of the algebraic space associated to the lattice.
The non-comparable elements $\theta,\d$ give rise to the sublattice $\set{\theta, \delta , \theta \v \delta, \theta \w \delta}$ which are called diamonds in the literature \cite{GLdef,glnext,HL,b-l}. These are closely related to the rank three embedded sublattice of the lattice $\L$. Where a sublattice $D \subset \L$ is called embedded if whenever $\a \v \b $ and $\a \w \b$ is in the lattice $D$ we have $\a , \b \in D$. Where rank is defined as the length of a maximal chain in the lattice. Counting occurrences of special types of sublattices is a recurrent theme in combinatorics. In the articles \cite{hexa, Subgroup, Dedikind} the problems of counting specific number of sublattices are studied. In various streams of mathematics namely in geometry and algebra these kinds of problems arise out naturally.

In the present article we introduce the problem of counting the number of non-comparable elements in a distributive lattice $\L$, which we denote by $n(\L)=1/2(\card{\set{\theta, \delta \in \L| \theta \nsim \delta }})$, where $\theta \nsim \delta$ denote that $\theta $ and $\delta $ are non-comparable to each other. The problem in full generality is to find a closed formula for the number $n(\L)$ for a lattice $\L$. Here we take up the problem partially by giving a set of lower and upper bounds for finite distributive lattices. Further as an application we look at the number of equations needed to define a variety which is an interesting numerical invariant of an algebraic space. We prove that the minimum number of defining equations for the algebraic variety associated to the lattice $\L$ is precisely the number of diamonds in a distributive lattice $\L$, which we denoted by $n(\L)$, is an interesting combinatorial invariant of the lattice. We show that $n(\L)$ is the minimal number of generators needed to define the variety $X(\L)$ in $\mathbb{A}^{\card{\L}}$. As an application we use the bounds developed in the article to find the necessary and sufficient conditions on lattices $\L$ for which the affine varieties $X(\L)$ are complete intersections, see \cite[p.~14]{Ha} for a definition of complete intersection. We have also added a conjectural lower bound of $n(\L)$ in the paper. For an introduction to the study of these varieties the articles \cite{HL,b-l,GLdef} and the book \cite{lakhi} are suggested.

A general introduction to distributive lattice theory can be obtained from the classics of the trade such as \cite{GG, Birkhoff}. A thorough geometric significance of these lattice theoretic structures can be obtained in \cite{Hibi, lakhi}.  From the point of view of a geometer one studies the ideal $I(\L)= \langle x_\a x_\b - x_{\a\v \b}x_{\a \w \b} | \a, \b \in \L \rangle $ in the polynomial algebra $k[\L]=k[x_\a | \a \in \L]$ and the vanishing locus of the said ideal in the affine space $\mathbb{A}^{\card{\L}}$. In the article \cite{Hibi} it was shown by Hibi that this binomial ideal is prime if and only if $\L$ is distributive, furthermore the $k$-algebra $k[\L]/I(\L)$ is normal when $\L$ is distributive. The vanishing locus of these binomial ideals in the affine space $\mathbb{A}^{\card{\L}}$ is a normal toric variety see the celebrated book by Sturmfels \cite{sturmfels,ES} for detail. The authors in \cite{glnext} related these toric varieties to a degeneration of Schubert varieties in $G/P_d$ by using a Grobner basis technique. Further they also calculated the cone and the faces of the toric varieties associated to $\L$. In the same article it was also shown that the dimension of the toric variety is equal to the rank of the distributive lattice. In the article \cite{wagner} Wagner relates the singular locus of these varieties to the poset of join irreducible $J$ of the lattice $\L$ using interesting lattice theoretic tools. The article \cite{HL} studied the singular locus of these varieties using a standard monomial theoretic approach they also find a standard monomial basis for the tangent cone in \cite{thesis}.

This present article is organised in four sections, apart from the introduction we have a separate section for all the main results of the paper in one place with references to appropriate lemmas and theorems for the proof. The next two sections are named as "Lower Bound" and "Upper Bound". In the section named "Lower Bound" we developed all the constructions and proved all the essential lemmas required to derive the lower bounds. Similarly we have used the section "Upper Bound" to prove all the lemmas required to derive the upper bounds. We have also included the proofs of the main theorems in either of these two sections depending on whether it is an upper bound or a lower bound.

\section{Main Results}

\begin{thm}For a distributive lattice $\L$ we have $n(\L) \geq \card{L} - \card{J}$.
\end{thm}
 We reduce the estimation of the number $n(\L)$ to the estimation of $f(\d)$ by a simple formula derived in the lemma \ref{equation}.  As a consequence of this result we become interested to know when the bound is actually achieved. In the following result we exactly identify the lattices for which we have the equality. In that aim we also develop a concatenation move in a distributive lattice in \ref{decomposition}. See the theorem \ref{lower bound} for the complete proof.

\begin{thm}For a distributive lattice $\L$ we have $n(\L)= \card{\L} - \card{J}$ if and only if the lattice $\L$ is isomorphic to a concatenation of copies of a diamond and chains of arbitrary lengths.
\end{thm}

We first simplified the lattice in question to a thick lattice see \ref{NS} for the definition. Then using the result for the thick lattices we prove the general case that the bound is achieved in the case of concatenations of chains.

Recall that an algebraic variety $V(f_1,f_2,f_3, \ldots, f_n)$ cut out by the set of polynomials $f_1,f_2,\ldots,f_n \in k[x_1,x_2,\ldots x_m]$ is called a complete intersection if the codimension of the variety as a subspace of $\mathbb{A}^m$ is exactly the number of the polynomials i,e equal to $n$. See \cite{Ha,eis} for a detailed introduction of algebraic varieties and the theory of complete intersections. As a result of the previous main theorem we immediately write down the only distributive lattice varieties that are complete intersections. See the theorem \ref{equality} for the complete proof.

\begin{thm} The affine variety $X(\L)$ is a complete intersection if and only if the lattice $\L$ is a concatenation of diamonds and chains.
\end{thm}
at this juncture it is a natural question to ask for an upper bound of the number $n(\L)$ the question, which we deal with in the next theorem. In the next theorem we observe that $f(\d) \geq \card{J}$ see the lemma \ref{approx f(d)} which gives us an upper bound for the number $n(\L)$. See theorem \ref{complete intersection} for the complete proof.

\begin{thm}$n(\mathcal{L}) \leq (\card{\mathcal{L}}-\card{J})\card{\mathcal{L}}/2$.
\end{thm}
See  theorem \ref{first approximation} for the complete proof.

\section{A Lower Bound}

\subsection{Definitions and Lemmas}
\begin{defn}\label{covers} A pair of elements $\a \ge \b$ in the lattice $\L$ are called covers or $\a$ covers $\b$ or $\b$ is covered by $\a$ if for every lattice point $z \in \L$ and $z \neq \b$ if $\a \geq z \ge \b$ then we have $z =\a$.
\end{defn}

\begin{defn} \label{NS}
A distributive lattice $\L$ is called a \textit{thick} lattice if for every element $\a$ which is not the maximal or the minimal element of the lattice, we have a lattice point $\b_\a$ such that $\b_\a \nsim \a$
\end{defn}

Let us define concatenation of two distributive lattices $\L_1,\L_2$. Let us assume that the maximal element of $\L_1$ is $\hat{1}$ and the minimal element of $\L_2$ is $\hat{0}$. Now let us define $\L_1 \# \L_2$ as the set $\L_1 \cup \L_2/ \sim $ where $ a \sim b $ if and only if either $a=b$ or $a=\hat{1}$ and $b=\hat{0}$. Note that this relation is clearly an equivalence relation. One can give a partial order on the equivalence classes as $[a] \geq [b] $ if $ a, b \in \L_1 $ and $a \geq b$ or $a \in \L_2 $ and $b \in \L_1$ or $a , b \in \L_2 $ and $ a \geq b$. It is not difficult to see that $\L_1 \# \L_2$ is a distributive lattice if and only if $\L_1$ and $\L_2$ are distributive.

\begin{defn} A diamond in a lattice $\L$ will be the set $D=\{a,b,c,d\}$ where $D$ is a sublattice and there are two elements in $D$ which are non-comparable. In other words if $\a,\b$ are the two non-comparable elements in $D$ then we can rewrite $D$ as the set $\{ \a,\b , \a \v \b, \a \w \b\}$.
\end{defn}

\begin{defn} To a distributive lattice $\L$ with $\card{\L} \geq 2 $ and a maximal join irreducible $\a$ we can associate a subset $\L_\a=\L \setminus \{\b \in \L | \b \geq \a\}$ which we will call the pruned subset.
\end{defn}

\begin{lem}\label{join prime} Let $\b \in J$ be a join irreducible in a lattice $\L$, then for all $\theta , \delta \in \L$ if $\theta \v \delta \ge \b$ then at least one of $\theta$ or $\delta$ is larger than $\b$.
\end{lem}
\begin{proof} Since $\theta \v \delta \ge \b$ we have $\b \in I_{\theta } \cup I_{\delta}$. So either $\b \in I_{\theta}$ or $\b \in I_{\delta}$. Equivalently we have either $\b \le \theta$ or $\b \le \delta$.
\end{proof}

\begin{thm} The pruned subset $\L_a$ is a sublattice and $J\setminus \{\a\}$ is its set of join irreducibles.
\end{thm}
\begin{proof}
Let $\theta , \delta \in \L_\a$ if $\theta \v \delta \notin \L_\a$ then $\theta \v \delta \geq \a$, since $\a$ is a join irreducible either $\theta $ or $\delta $ is larger than $\a$ ( see lemma \ref{join prime} ) which is a contradiction to our assumption about $\theta $ and $\delta$. If $\theta \w \delta \notin \L_\a$ then $\theta \w \delta \geq \a $ which means both $\theta $ and $\delta$ is larger than $\a$ another contradiction to our assumption. So both $\theta \v \delta $ and $\theta \w \delta$ are in $\L_\a$.
Now for the second part of the statement note that the join irreducibles of $\L$ other than $\a$ remains to be join irreducible in $\L_\a$ and $\a$ being a maximal join irreducible there are no join irreducible larger than $\a$. Now we just have to show that there are no other join irreducible in $\L_\a$, which follows from the fact that $\L_\a$ being a distributive lattice the size of its set of join irreducibles is equal to the length of one of its maximal chain. Now since the maximal element of $\L$ say $\hat{1}$ is not in $\L_\a$ the length of a maximal chain in $\L_\a$ is smaller than that of $\L$ by at least one. But since we have already produced $\card{J} -1$ join irreducibles we have exact equality. or the set of join irreducibles of $\L_\a$ is $J \setminus \{\a\}$.
\end{proof}

\begin{defn} Let $\d \in \L$ we will denote the set $\{\a \in \L | \a \sim \d \}$ by $\Lambda_\delta$.
\end{defn}

\begin{lem} $\Lambda_\delta$ is a sublattice of $\L$ for all $\d$ and it is equal to $[\hat{0},\d]\cup[\d,\hat{1}]$ where $[a,b]=\{ x \in \L | a \leq x \leq b \}$.
\end{lem}
\begin{proof} Let us first prove that $\Lambda_\d=[\hat{0},\d]\cup[\d,\hat{1}]$. If $x \in \Lambda_\d$ then $x \sim \delta$ hence $x \leq \d$ or $x \geq \d$ or equivalently $x \in [\hat{0},\d]$ or $x \in [\d,\hat{1}]$. Thus we have $\Lambda_\d \subset [\hat{0},\d]\cup[\d,\hat{1}]$. For the other inclusion, let $x \in [\hat{0},\d]\cup[\d,\hat{1}]$ then clearly $x$ is comparable to $\d$. So $\Lambda_\d=[\hat{0},\d]\cup[\d,\hat{1}]$. Now let us prove that $[\hat{0},\d]\cup[\d,\hat{1}]$ is a sublattice. Let $x,y \in [\hat{0},\d]\cup[\d,\hat{1}]$ then we have four cases to consider according as listed below.
\begin{enumerate}
\item Both $x,y \in [\hat{0},\d]$ then since $x,y$ are both larger than $\d$ then both $x \v y$ and $x \w y$ are larger than $\d$. So $x \v y , x \w y \in [\hat{0},\d]$.
\item Both $x,y \in [\d,\hat{1}]$ then similar to the above case we have both $x,y$ smaller than $\d$ which means both $x \v y, x \w y$ are smaller than $\d$.
\item $x \leq \d$ and $y \geq \d$ in this case $y \geq x$ so $ x \v y = y $ and $x \w y =x$ so we have both $x \v y, x \w y \in \Lambda_\d$.
\item Similar to the above case we have $y \leq \d$ and $x \geq \d$.
\end{enumerate}
This proves that $\Lambda_\d$ is a sublattice of $\L$.
\end{proof}

\begin{defn} Let us define a function $f \colon \L \longrightarrow \mathbb{N}$ as $f(\d)= \card{\Lambda_\d}$.
\end{defn}

\begin{lem} A distributive lattice $\L$ with more than one element is thick if and only if it is not isomorphic to concatenation of two lattices.
\end{lem}
\begin{proof}
We will prove the equivalence of the negations. If a lattice $\L$ is isomorphic to concatenation of the lattices let us say $\L=\L_1\#\L_2$ then the class of the minimal element $\hat{0}$ of $\L_2$ which is equal to the class of the maximal element $\hat{1}$ of $\L_1$ is comparable to all elements in both $\L_1$ and $L_2$. That proves that the lattice is not thick.
Now if the lattice $\L$ is not thick then there is a $\delta$ comparable to all elements in $\L$. So let $\L_1=\{\a \in \L | \a \geq \delta \}$ and $\L_2=\{ \a \in \L | \delta \geq \a \}$. $\L= \L_1 \cup \L_2$. Let us equip the set $\L$ with trivial equivalence relation $\sim $ that is $\a \sim \b $ if and only if $\a = \b$ then we see that $\L \simeq \frac{\L_1 \cup \L_2}{ \sim} \simeq \L_1 \# \L_2$; the first isomorphism since we have a quotient with respect to a trivial equivalence relation, and second isomorphism since this equivalence relation coincides with the concatenation relation in this case. This completes the proof.
\end{proof}

\begin{lem}\label{decomposition} For every distributive lattice $\L$, there are thick sublattices $\L_1,\L_2 ,\ldots , \L_n$ such that $\L$ is isomorphic to $\L_1 \# \L_2 \# \L_3 \#  \ldots \# \L_n$.
\end{lem}
\begin{proof}
We prove this statement by an induction argument on the number $\card{\L}$. For the base case, a lattice with a single point $\hat{0}=\hat{1}$ is thick by definition so we have nothing to prove. Now let us assume that we have the result for all lattices with size less than that of $\L$. if the lattice $\L$ is thick then we have nothing to prove, otherwise there is a lattice point $\d \in \L$ such that every lattice point $\a \in \L$ is comparable to $\d$. Consider $\L_1=\{ \a \in \L | \a \geq \d \}$ and $\L_2=\{\a \in \L | \a \leq \d \}$. We have $\L= \L_2 \cup \L_1$ and the minimal element of $\L_2$ is $\d$ which is also the minimal element of $\L_1$. Hence the concatenation $\L_2 \# \L_1$ is isomorphic to $\L$. Now by induction we can decompose both the lattices $\L_1$ and $\L_2$ into thick lattices.
\end{proof}

\begin{lem}\label{join lemma} For every covers $\a \ge \g \in \L$ if $\a$ is not a join irreducible we have a unique join irreducible $\b \in J$ such that $\a = \g \v \b$ and equivalently $I_a = \{\b\} \cup I_\g$, if $\a$ is a join irreducible then there is a unique join irreducible $\b$ such that $\g = \b \w \a$.
\end{lem}
\begin{proof}
Let $\a \geq \g$ be covers, consider $I=I_{\a} \ I_{\g}$. Let $\b_1,\b_2 \in I$ then consider $\g \leq \b_1 \v \g \leq \b_1 \v \b_2 \v \g \leq \a$ since $\a$ covers $\g$ we have $\b_1 \v \g = \b_1 \v \b_2 \v \g$ or $\b_2 \leq \b_1 \v \g$. But $\b_2 \ge \g$ so we must have $\b_2 \leq \b_1$. We can do the same calculation with $\g \leq \b_2 \v \g \leq \b_2 \v \b_1 \v \g \leq \a $ to get $\b_1 \leq \b_2$. Hence we have $\b_1 = \b_2$ and $I_{\a} = \{\b_1 \} \cup I_{\g}$ and equivalently $\a = \b_1 \v \g$. Or taking $\b = \b_1$ we have the result.
\end{proof}

\begin{thm} The ideal $I(\L) = \langle x_\a x_\b - x_{\a \v \b} x_{\a \w \b} | \a, \b \in \L \rangle $ in the polynomial algebra $k[\L]$ is optimally generated by $n(\L)$ elements.
\end{thm}
\begin{proof}Let $D_1,D_2, \ldots, D_n$ where $n=n(\L)$ be all the distinct diamonds of $\L$ and let us say $D_i = \{ \theta_i, \delta_i , \theta_i \v \delta_i , \theta_i \w \delta_i \}$.Let us also denote the diamond relations $f_{D_i}=x_{\theta_i}x_{\delta_i}-x_{\theta_i \v \delta_i}x_{\theta_i \w \delta_i}$ by $f_i$. If the ideal $I(\L)$ is not optimally generated let us say $f_1=x_{\theta_1}x_{\delta_1}-x_{\theta_1 \v \delta_1}x_{\theta_1 \w \delta_1}$ can be generated by $f_2,f_3, \ldots , f_n$.

So we have $\epsilon(i) \in k[\L]$ such that $f_i= \sum_{i \geq 2} \epsilon(i)f_i$. Now since $f_i$ are homogeneous polynomials of degree two, we have $\epsilon(i) \in k$. So $x_{\theta_1}x_{\delta_1} = x_{\theta_1 \v \delta_1}x_{\theta_1 \w \delta_1} + \sum \epsilon(i)f_i$. Which means the product $x_{\theta_1}x_{\delta_1}$ occurs in $f_i$ for some $i$. Or equivalently there is a diamond $D_i$ with $i >1$ such that $\theta_1, \delta_1 \in D_i=\{\theta_i , \delta_i , \theta_i \v \delta_i , \theta_i \w \delta_i \}$. Since we have $\theta_j \nsim \delta_j$ from the definition of a diamond, we must have $\{\theta_1, \delta_1\}=\{\theta_i , \delta_i\}$ i.e $D_1=D_i$ which contradicts our assumption that the diamonds $D_i$ are all distinct for distinct $i$.

\end{proof}

\begin{thm}\label{first}\label{equation}
For a distributive lattice $\L$ we have $n(\mathcal{L})=1/2(\card{\mathcal{L}}^2 - \displaystyle{\sum_{\d \in \L} f(\d)})$.
\end{thm}
\begin{proof}

\noindent
Let us fix a $\d \in \L$ and let us consider the set $\{\theta \in \L | \theta \nsim \d \}$. For each element $\theta$ from this set we have a diamond relation given by the non-comparable pair $(\theta, \delta)$. Clearly we have $\{\theta \in \L | \theta \nsim \d \}= \L \setminus \Lambda_\d$. So we have the following bijective sets, $\{(\theta, \delta) | \theta \nsim \delta \} \simeq \{\theta \in \L | \theta \nsim \d \}= \L \setminus \Lambda_\d$. Thus we have: \[\card{\{(\theta, \delta) | \theta \nsim \delta \}}= \card{\L} - \card{\Lambda_\d}.\]
we denote $\card{\Lambda_\d}$ by $f(\d)$ which lets us to rewrite the above equation as : \[\card{\{(\theta, \delta) | \theta \nsim \delta \}}= \card{\L} - f(\d).\]

Now let us sum both sides for all choices of $\d \in \L$ and considering the double counting of the non-comparable pairs we get \[ 2n(\L)=\displaystyle{\sum_{\d \in \L} \card{\L} - f(\d) }\] or \[ 2n(\L)= \card{\L}^2 -\displaystyle{\sum_{\d \in \L} f(\d) }\] or \[n(\mathcal{L})=1/2(\card{\mathcal{L}}^2 - \displaystyle{\sum_{\d \in \L} f(\d)}).\]

\end{proof}

\begin{lem}\label{approx f(d)} $f(\d) \geq \card{J}$ for all $\d \in \L$.
\end{lem}
\begin{proof} $f(\d) = \card{\Lambda_\d}$ now since $\Lambda_\d$ will contain a maximal chain of $\L$ through $\d$ we have $\card{\Lambda_\d} \geq \card{\mathcal{M}_\d}$, where $\mathcal{M}_\d$ is a maximal chain through $\d$. Since all maximal chains of $\L$ have length equal to $\card{J}$ we get the result.
\end{proof}

\subsection{The inequalities}
Let us conclude this section with a lower bound for the numbers $n(\L)$ in the theorem below.
\begin{thm} \label{lower bound} For a distributive lattice $\L$ we have $n(\L) \geq \card{L} - \card{J}$.
\end{thm}
\begin{proof}For each $\theta \in \L \setminus J$ let us associate a diamond relation as $(\a_1,\a_2)$ where $\a_1,\a_2$ is a pair of elements given by property that $\a_1 \v \a_2 =\theta$ there exist such pair since we have chosen $\theta $ from the set of elements outside the poset of join-irreducibles. There could be more than one such pair so we chose one in such cases on ambiguity. Note that this association is one to one so we have $\card{\L \setminus J} \leq \card{ \{(\a,\b) | \a \nsim \b \}}$. Hence $n(\L) \geq  \card{L} -\card{J}$.
\end{proof}

\begin{lem}\label{forward} If $\L$ is thick and $\L$ is not a diamond then $n(\L) > \card{\L} -\card{J}$.
\end{lem}
\begin{proof} Since $\L$ is not a diamond we have $\card{\L \setminus J}=\card{\L} - \card{J} \geq 2$. We will prove the lemma with an induction argument on the number $\card{\L} - \card{J}$. For the base case we have $\card{\L} - \card{J} =2$, let the two elements be $\a_1, \a_2 \in \L \setminus J$. Note that one of these two is the maximal element of the lattice namely $\hat{1}$, without a loss of generality let us assume that $\a_1 = \hat{1}$. And let us assume that $\a_i = \theta_i \v \delta_i$ for $i=1,2$ and $\theta_i \nsim \delta_i$. If both $\theta_1$ and $\delta_1$ are comparable to $\a_2$ then since the lattice $\L$ is thick there must be a $\g \in \L$ such that $\a_2$ and $\g$ are non-comparable. This gives a new diamond relation unless $\{\g, \a_2, \g \v \a_2 , \g \w \a_2\}$ is equal to $\{\theta_1, \delta_1, \theta_1 \w \delta_1 , \a_1\}$ which is not possible since $\a_2$ has been assumed to be comparable to both $\theta_1$ and $\delta_1$. In other case when $\a_2$ is non-comparable to either of $\theta_1$ or $\delta_1$ then we have a new diamond relation associated to that non-comparable pair. So in both the cases we have number of diamond relations $n(\L)$ larger than $\card{\L \setminus J}=2$.

For the general case of the induction let us assume that we have the result for all the lattices with cardinality less than that of $\L$. Let $\a$ be a maximal join irreducible element. Let us consider pruned lattice with respect to this maximal join irreducible (see \cite{HM} for a thorough treatment of pruned lattices ) $\L_{\a}=\L \setminus K_{\a}$ where $K_{\a}= \{ \g \in \L | \g \geq \a \}$. We know that $\L_\a$ is a sublattice of $\L$ with $J \setminus \{\a\}$ as it's set of join irreducibles. Applying the result to this lattice and observing that the size of the set of join irreducibles of the sublattice $\L_\a$ is one less than the size of $J$, we have $n(\L_\a) \ge \card{\L_\a}-\card{J}+1$. Now for each $b \in K_\a$ if $b \neq \a$ then there are $\theta, \delta \in \L$ such that $b = \theta \v \delta$. And note that not both of these are in $\L_\a$ since that will mean $b= \theta \v \delta \in \L_\a$ contradicting our assumption. So the diamond $\{ b , \theta , \delta , \theta \w \delta \}$ is not counted in $n (\L_\a)$. Or $n(\L) - n (\L_\a) \geq \card{K_\a} -1 = \card{\L} - \card{L_\a} -1$, which gives $n(\L) \geq n(\L_\a) + \card{\L} - \card{L_\a} -1$. Applying the induction assumption we have $n(\L) \geq \card{\L_\a} - \card{J} +1 + \card{\L} -\card{\L_\a} -1 = \card{\L} - \card{J}$.
\end{proof}

\begin{thm} \label{equality} For a distributive lattice $\L$ we have $n(\L)= \card{\L} - \card{J}$ if and only if the lattice $\L$ is isomorphic to a concatenation of copies of a diamond and chains of arbitrary lengths.
\end{thm}
\begin{proof}From lemma \ref{decomposition} we know that every distributive lattice $\L$ is decomposable into a concatenation of sublattices which are thick so we will assume that the lattice $\L$ is thick. For we see clearly that $n(\L_1 \# \L_2)= n (\L_1) + n ( \L_2)$. Now by the lemma \ref{forward} we know that $\L$ has to be a diamond.
\end{proof}

\begin{thm} \label{complete intersection} The affine variety $X(\L)$ is a complete intersection if and only if the lattice $\L$ is a concatenation of diamonds and chains.
\end{thm}
\begin{proof} We know the dimension of the variety $X(\L)$ is $\card{J}$, so the minimal number of generators for the ideal $I(\L)$ i.e $n(\L)$ is equal to the codimension $\card{\L} - \card{J}$ if and only if $\L$ is concatenation of chains and diamonds by the theorem \ref{equality}. So we have the result.
\end{proof}

Let us define a notation for the number of covers (or equivalently edges in the Hasse graph ) in the lattice $\L$ as $e(\L)$. With this notation in mind we have the following theorem.

\begin{thm} \label{edge approx} $n(\L) \geq e(\L) - \card{\L} +1$.
\end{thm}
\begin{proof} Let us take a spanning tree $T$ in the Hasse graph of the lattice $\L$. There are $\card{\L} -1$ edges in the tree $T$. Let  $A$ be the collection of edges in $\L$ $(\theta, \delta)$ which are not on the spanning tree $T$. Since $(\theta, \delta) \in A$ is an edge we have $\theta \gtrdot \delta $. So by \ref{covers} we have a unique join-irreducible $\a$ such that $\theta = \delta \v \a $. So we have a diamond $\{\theta , \delta , \delta \w \a , \a \}$. We have different such diamonds for different edges $(\theta, \delta)$. Since otherwise if $(\theta, \delta )$ and $(x,y)$ gives rise to same diamond $\{\theta , \delta , \a , \delta \w \a \}$
\end{proof}

Here let us also include a conjectural lower bound for the numbers $n(\L)$.
\begin{conj} For thick distributive lattices $\L$ we have $n(\L) \geq e (\L) - \card{J}$, where $J$ is the poset of join irreducibles of $\L$.
\end{conj}


\section{An Upper Bound}


It is a natural question now to ask for an upper bound for the number $n(\L)$. In this section we will deal with this question. In the following theorem we write down a bound as immediate consequence of the lemma \ref{approx f(d)}.
\begin{thm}\label{first approximation} $n(\mathcal{L}) \leq (\card{\mathcal{L}}-\card{J})\card{\mathcal{L}}/2$.
\end{thm}
\begin{proof} From lemma \ref{equation} we have $n(\mathcal{L})=(\card{\mathcal{L}}^2 -\displaystyle{\sum_{\d \in \L} f(\d)})/2$. And from lemma \ref{approx f(d)} we have $\displaystyle{\sum_{\d \in \L} f(\d) \geq \card{\L}\card{J}}$. Putting these two together we have \[n(\L) \leq (\card{\L}^2 -\card{\L}\card{J})/2.\] Or \[n(\L) \leq \card{\L}(\card{\L} -\card{J})/2.\]
\end{proof}



\bibliographystyle{abbrv}

\end{document}